\documentclass[a4paper,10pt]{article}

\usepackage{graphicx}
\usepackage{geometry}
\usepackage{amsthm}
\usepackage{amssymb}
\usepackage{amsmath}
\usepackage{hyperref}
\usepackage{xcolor}
\usepackage{enumerate}
\usepackage{listings}
\usepackage{complexity}
\usepackage{blkarray}
\usepackage{lmodern}

\usepackage[font=small,labelfont=bf]{caption}
\usepackage[font=small,labelfont=normalfont,labelformat=simple]{subcaption}
\graphicspath{{figs/}}

\newtheorem{theorem}{Theorem}
\newtheorem{definition}{Definition}
\newtheorem{lemma}[theorem]{Lemma}
\newtheorem{corollary}[theorem]{Corollary}

\newtheorem{observation}{Observation}
\newtheorem{proposition}{Proposition}
\newtheorem{conjecture}{Conjecture}
\newtheorem{problem}{Problem}
\newtheorem{remark}{Remark}

\newcommand{\revvec}[1]{\overset{\text{\tiny$\longleftarrow$}}{#1}}
\newcommand{\bivec}[1]{\overset{\text{\tiny$\longleftrightarrow$}}{#1}}

\author
{
Raphael Steiner \thanks{Department of Mathematics, ETH Z\"{u}rich, Switzerland,  \texttt{raphaelmario.steiner@math.ethz.ch}.
Supported by the SNSF Ambizione Grant No. 216071.}
}
\date{\today}

\title{Openly disjoint cycles and directed tree-width\\ of regular digraphs}

\begin{document}
\maketitle

\begin{abstract}
Given a digraph $D$, let $c(D)$ denote the largest integer $k$ such that there exists a collection of $k$ openly disjoint cycles through a vertex, i.e., a collection of directed cycles $C_1,\ldots,C_k$ all containing a common vertex $v$ such that $C_1-v,\ldots,C_k-v$ are pairwise vertex-disjoint. The famous Caccetta-Häggkvist conjecture and its regular variant due to Behzad, Chartrand and Wall from 1970, have motivated the study of degree conditions forcing $c(D)$ to be large.

Surprisingly, in 1985 Thomassen constructed digraphs of arbitrarily high minimum out- and in-degree such that $c(D)\le 2$. In 2005, Seymour asked whether in contrast every $r$-regular digraph satisfies $c(D)=r$, which would have implied the aforementioned conjecture of Behzad, Chartrand and Wall. In 2008, Mader answered this negatively for every $r\ge 8$, but conjectured that nevertheless the minimum value $c_r$ of $c(D)$ over all $r$-regular digraphs grows with $r$, i.e. $\lim_{r\rightarrow\infty}c_r=\infty$. 

As the first main result of our paper, we prove Mader's conjecture in a strong form by showing that $c_r\ge \lceil\frac{3}{22} r\rceil$ for every $r\in \mathbb{N}$. We also show that 
$c_r\le 7\left\lceil \frac{r}{8}\right\rceil$, which improves the previous best upper bound $c_r\le r-\Theta(\sqrt{r})$ due to Mader.

In the second main result of this paper we use the same proof technique to show that every $r$-regular digraph has directed tree-width at least $\Omega(r)$. This bound is tight up to the implied constant and cannot be extended to digraphs of minimum out- and in-degree at least $r$, where we point out that no growing lower bound on the directed tree-width can be guaranteed. As a corollary we obtain the existence of a function $f:\mathbb{N}\rightarrow \mathbb{N}$ such that every regular digraph with degree at least $f(k)$ contains a subdivision of the cylindrical wall of order $k$, and hence of a large class of planar digraphs. This represents new progress on the notoriously difficult problem of finding degree conditions guaranteeing subdivisions of digraphs, going back to another, well-known, conjecture of Mader from 1985.
\end{abstract}

\section{Introduction}
The infamous \emph{Caccetta-Häggkvist-conjecture}, proposed by Caccetta and Häggkvist~\cite{MR527946} in 1978, is one of the most important open problems in all of extremal graph theory. Recall that the \emph{girth} of a digraph is defined as the length of its shortest directed cycle.
\begin{conjecture}[Caccetta and Häggkvist 1978~\cite{MR527946}]\label{con:1}
Let $r,g\in \mathbb{N}$. Every digraph $D$ with minimum out-degree $\delta^+(D)\ge r$ and girth at least $g$ contains at least $r(g-1)+1$ vertices.
\end{conjecture}

Despite enormous efforts put into attacking this conjecture by many leading researchers in the past, it remains widely open nowadays, see the survey~\cite{sullivan2006} for a (by now $20$ years old) summary of partial results and related problems. Strikingly, the conjecture even remains open when $g=4$, see~\cite{MR3638333} for the best known partial result in this case. The conjecture also remains wide open in the important special case of $r$-regular digraphs (i.e., digraphs where every vertex has out- and in-degree exactly $r$). In fact, this special case was separately conjectured already 8 years earlier by Behzad, Chartrand and Wall:

\begin{conjecture}[Behzad, Chartrand and Wall 1970,~\cite{MR285448}]\label{con:2}
Let $r,g\in\mathbb{N}$. Every $r$-regular digraph with girth at least $g$ has at least $r(g-1)+1$ vertices.
\end{conjecture}

A very natural approach towards Conjectures~\ref{con:1},~\ref{con:2} is to try and show that every digraph of minimum out-degree at least $r$ (or every $r$-regular digraph) contains a sequence of $r$ directed cycles with very small overlaps. Concretely, the following so-called \emph{Hoang-Reed conjecture} from 1978 still remains open and would directly imply the Caccetta-Häggkvist conjecture.

\begin{conjecture}[Hoang and Reed~1978,~\cite{MR900933}]
Let $r\in \mathbb{N}$. Every digraph $D$ with $\delta^+(D)\ge r$ contains a sequence $C_1,\ldots,C_r$ of directed cycles such that for each $i\in [r]$ it holds that $$\left|V(C_i)\cap \bigcup_{1\le j<i}V(C_j)\right|\le 1.$$
\end{conjecture}

A particularly simple instance of a collection of directed cycles satisfying the constraints in the Hoang-Reed conjecture is a collection of \emph{openly disjoint cycles through a vertex} in a digraph $D$. Concretely, this is defined as a collection $C_1,\ldots,C_k$ of directed cycles in $D$ such that there exists a vertex $v\in V(D)$ common to all cycles $C_1,\ldots,C_k$ and such that $V(C_1)\setminus \{v\},\ldots, V(C_k)\setminus \{v\}$ are pairwise disjoint. Throughout the rest of this paper and following the notation in~\cite{MR2588727,MR2672214}, given a digraph $D$ we denote by $c(D)$ the maximum size of a collection of openly disjoint directed cycles through a vertex in $D$. A natural refinement of the Hoang-Reed conjecture is then to ask whether every digraph $D$ with $\delta^+(D)\ge r$ satisfies $c(D)\ge r$. Thomassen showed that this is indeed the case when $r=2$. \begin{theorem}[Thomassen~1985,~\cite{MR793632}]
Every digraph $D$ with $\delta^+(D)\ge 2$ satisfies $c(D)\ge 2$. 
\end{theorem}
However, and quite surprisingly, this approach completely breaks down when considering digraphs of minimum out- and in-degree at least $r\ge 3$, as Thomassen showed in another paper:

\begin{theorem}[Thomassen 1985~\cite{MR793491}]\label{thm:thomassenconstruction}
For every $r\in\mathbb{N}$ there exists a digraph $D$ with $\delta^+(D),\delta^-(D)\ge r$ but $c(D)\le 2$.
\end{theorem}

While this destroys any hope of proving the Caccetta-Häggkvist conjecture by lower-bounding $c(D)$, the digraphs constructed by Thomassen to prove Theorem~\ref{thm:thomassenconstruction} are highly irregular. As in many directed graph problems regular digraphs are significantly better behaved, this led Seymour to pose the natural problem whether it could still be possible to tackle the Behzad-Chartrand-Wall-conjecture via a corresponding lower bound on $c(D)$.
\begin{problem}[Seymour~2005, cf.~\cite{MR2588727,MR2672214}]\label{prob:seym}
Does every $r$-regular digraph $D$ satisfy $c(D)\ge r$?
\end{problem}

Mader studied this problem extensively in~\cite{MR2588727,MR2672214}. Following his notation and to make the presentation in the following easier, we will from now on denote by $c_r$ the minimum of $c(D)$ taken over all $r$-regular digraphs $D$. Observe that trivially, we always have $c(D)\le r$. Hence, Seymour's question rephrases as asking whether $c_r=r$ for every $r$.

Interestingly, Mader managed to prove that Problem~\ref{prob:seym} has a positive answer when $r=3$, distinguishing regular digraphs from the negative result for general digraphs in Theorem~\ref{thm:thomassenconstruction}. Unfortunately however, Mader also constructed, for every $r\ge 8$, an $r$-regular digraph $D$ with $c(D)\le r-1$, showing that $c_r\le r-1$ for every $r\ge 8$, and hence the answer to Problem~\ref{prob:seym} is negative for $r\ge 8$. It remains open in all the cases $r\in \{4,\ldots,7\}$. On the positive side, Mader~\cite{MR2588727,MR2672214} proved that Seymour's problem has a positive answer if one additionally assumes the digraph at hand to have strong vertex-connectivity at least $r-2$ or assumes it to be vertex-transitive. He also showed that $c_r\ge 4$ for every $r\ge 7$, but could not prove any growing lower bound on $c_r$ or that $c_r\ge 5$ for any $r$. Mader did however conjecture~\cite{MR2588727} that $c_r$ should grow arbitrarily large as $r\rightarrow \infty$ and reiterated this conjecture in~\cite{MR2672214}. 

\begin{conjecture}[Mader~2008,~\cite{MR2588727}]
For every $k\in\mathbb{N}$ there exists some $r_0\in\mathbb{N}$ such that $c_r\ge k$ for every $r\ge r_0$. In other words, $\lim_{r\rightarrow\infty}c_r=\infty$.
\end{conjecture}
As the first main result of this paper, we confirm this conjecture in a strong form: We show that $c_r$ in fact grows linearly in $r$, which is asymptotically best possible.
\begin{theorem}\label{thm:main}
Let $r\in\mathbb{N}$. Every $r$-regular digraph $D$ satisfies $c(D)\ge \left\lceil\frac{3}{22} r\right\rceil$. 
\end{theorem}
Theorem~\ref{thm:main} can be seen as affirming Problem~\ref{prob:seym} in a qualitative sense. We also remark that it is known that the Caccetta-H\"{a}ggkvist-conjecture holds up to constant factors due to a classic result of Chv\'{a}tal and Szemer\'{e}di~\cite{MR735200}, namely every digraph with minimum out-degree at least $r$ and girth $g$ has at least $\Omega(rg)$ vertices. Since every digraph $D$ of girth $g$ satisfies $\mathrm{v}(D)\ge c(D)(g-1)+1$, Theorem~\ref{thm:main} offers a short alternative proof of this fact in the case of regular digraphs, albeit with a worse constant than that obtained in~\cite{MR735200}.

Given Theorem~\ref{thm:main}, it is natural to wonder how close to the ideal bound of $r$ the value of $c_r$ is in general. Mader~\cite{MR2588727} studied this question and showed that the difference $r-c_r$ grows at least as fast as $\sqrt{r}$.
\begin{theorem}[Mader 2008,~\cite{MR2588727}]\label{thm:maderconstruction}
For every $k\ge 2$ there exists a $2k^2$-regular digraph $D$ with $c(D)\le 2k^2-k+1$.
\end{theorem}
This implies that $c_r\le r-\left(\frac{1}{\sqrt{2}}-o(1)\right)\sqrt{r}$ for every $r$. Given this bound, one may still hope that $c_r=r-o(r)$. Unfortunately, this is not the case: Using blow-ups, we demonstrate that $r-c_r$ grows linearly in $r$, hence improving the aforementioned bound of Mader. Concretely, we first show the following fact about the numbers $c_r$, which then directly implies a linear separation of $c_r$ from $r$ using that $c_8\le 7$.
\begin{proposition}\label{prop:upper}
For every $r, b\in \mathbb{N}$, we have $c_{rb}\le c_r\cdot b$.
\end{proposition}

\begin{corollary}\label{cor:upper}
    For every $r\in \mathbb{N}$, we have $c_r\le 7\left\lceil \frac{r}{8}\right\rceil$.
\end{corollary}
\begin{proof}
As mentioned before (and follows also from Theorem~\ref{thm:maderconstruction} for $k=2$), we have $c_8\le 7$. Hence, Proposition~\ref{prop:upper} implies that $c_{8b}\le 7b$ for every $b\in \mathbb{N}$. Since it is furthermore easy to see (and was observed by Mader~\cite{MR2588727}) that $c_{i}\le c_j$ for every $i\le j$, this implies that for every $r\in \mathbb{N}$ we have
$$c_r\le c_{8\left\lceil\frac{r}{8}\right\rceil}\le 7\left\lceil \frac{r}{8}\right\rceil,$$ as desired.
\end{proof}

Interestingly, the technique we use for our proof of Theorem~\ref{thm:main} (described in more detail at the end of this section) can also be used to prove a new result about a seemingly unrelated topic, namely \emph{directed tree-width}. The latter is the natural extension of the famous and versatile parameter \emph{tree-width} for undirected graphs and was introduced by Johnson, Robertson, Seymour and Thomas~\cite{MR1828440} in 2001. Given a digraph $D$, we denote by $\mathrm{dtw}(D)\in \mathbb{N}_0$ its directed tree-width. Intuitively, directed-tree-width is a measure of the structural similarity of a digraph $D$ to an acyclic digraph (the latter are exactly those digraphs for which $\mathrm{dtw}(D)=0$). Since its introduction, directed tree-width has found many applications in the structural theory of directed graphs as well as for algorithmic problems such as parameterized algorithms as well as routing and linkage problems on digraphs. 

For undirected graphs, it is well-known that every graph of minimum degree at least $r$ has tree-width at least $\Omega(r)$. In fact, this follows directly from the following bound on the number of edges in graphs with given tree-width.
\begin{lemma}[Folklore]\label{lem:twedgebound}
Every graph with $n$ vertices and tree-width $t$ has at most $$t(n-t)+\binom{t}{2}\le tn$$ edges.
\end{lemma}

It is natural to suspect an analogous statement to be true for directed graphs, i.e., that every digraph $D$ with minimum out- (and/or) in-degree at least $r$ has directed tree-width at least $\Omega(r)$. A statement of this type would have some desirable consequences (as we discuss further below). Unfortunately, this turns out to be (very) false.

\begin{remark}\label{rem:outindtw}
For every $r\in \mathbb{N}$ there exists a digraph $D$ with $\delta^+(D), \delta^-(D)\ge r$ and $\mathrm{dtw}(D)=1$. 
\end{remark}
\begin{proof}
Let $r\in \mathbb{N}$ be given. It was shown by the author in~\cite[Remark~2.1]{MR4434351} that there exists a digraph $F_r$ with $\delta^+(F_r)=r$ and $\mathrm{dtw}(F_r)=1$. 
Let $\revvec{F}_r$ denote the digraph obtained from $F_r$ by reversing the directions of all arcs. It is easy to verify (cf.~\cite{MR1828440}) that the directed-tree width remains unaffected by the operation of reversing all arcs, and hence we also have $\mathrm{dtw}\left(\revvec{F}_r\right)=1$. We further have $\delta^-\left(\revvec{F}_r\right)=\delta^+(F_r)=r$. Let now $D$ be defined as the digraph obtained from the disjoint union of $F_r$ and $\revvec{F}_r$ by adding all possible arcs from $\revvec{F}_r$ to $F_r$, i.e., all arcs in the set $V\left(\revvec{F}_r\right)\times V(F_r)$. It is easy to verify that we now have $\delta^+(D),\delta^-(D)\ge r$. Furthermore, it follows by definition that the strongly connected components of $D$ are exactly those of $F_r$ together with those of $\revvec{F}_r$. Since it is well-known and easy to verify (cf.~\cite{MR1828440}) that the directed tree-width of a digraph $D$ equals the maximum of the directed tree-width of its strongly connected components, it follows that $\mathrm{dtw}(D)=\max\left\{\mathrm{dtw}(F_r),\mathrm{dtw}\left(\revvec{F}_r\right)\right\}=\max\{1,1\}=1$. This concludes the proof of the remark.
\end{proof}

It can be checked that the construction of the digraphs in Remark~\ref{rem:outindtw} is highly irregular, and hence it is natural to hope that a lower-bound on the directed tree-width in terms of degrees can be recovered in the case of regular digraphs. As the second main result of this paper, we prove that this is indeed the case.

\begin{theorem}\label{thm:main2}
Let $r\in \mathbb{N}$. Every $r$-regular digraph has directed tree-width at least $\left\lfloor\frac{r}{20}\right\rfloor=\Omega(r)$.
\end{theorem}

The bound in Theorem~\ref{thm:main2} is optimal up to the constant factor, since the complete digraph $\bivec{K}_{r+1}$ on $r+1$ vertices and with all possible $r(r+1)$ arcs is $r$-regular and has directed tree-width~$r$. 

Theorem~\ref{thm:main2} has a pleasing consequence concerning subdivisions in digraphs, which we shall describe next. To do so, we need to recall the \emph{directed grid theorem} due to Kawarabayashi and Kreutzer~\cite{MR3388245}, one of the premier results of structural digraph theory, which extends Robertson and Seymour's famous \emph{grid minor theorem}~\cite{MR854606} to directed graphs. To state the result, we first define the so-called \emph{cylindrical walls}.

\begin{definition}[Cylindrical wall, cf.~\cite{MR4141197}]
The \emph{cylindrical wall of order $k$} for $k \in \mathbb{N}$ is the planar digraph $W_k$ with vertex-set $V(W_k)=[2k]\times [2k]$ and in which there is an arc from a vertex $(x,y)$ to another vertex $(x',y')$ if and only if one of the following holds (see Figure~\ref{fig:wall} for an illustration):
\begin{itemize}
    \item $y=y'$ is odd and $x'=x+1$.
    \item $y=y'$ is even and $x'=x-1$. 
    \item $x=x'$ is odd, $y$ is even and $y'\equiv y+1 \text{ (mod }2k)$.
    \item $x=x'$ is even, $y$ is odd and $y'\equiv y+1 \text{ (mod }2k)$.
\end{itemize}
\end{definition}

\begin{figure}
    \centering
\includegraphics[scale=0.7]{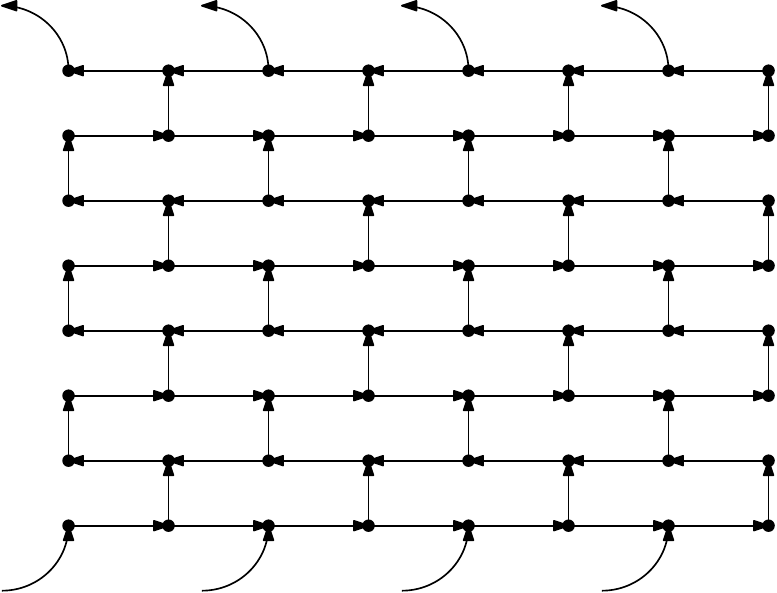}
    \caption{Illustration of the cylindrical wall of order $4$. The indicated half-arcs loop back from the top-row of the wall to the bottom-row (without crossings).}
    \label{fig:wall}
\end{figure}
\noindent Recall that  a \emph{subdivision} of a digraph $F$ is a digraph obtained from $F$ by replacing (a subset of) its arcs by internally disjoint directed paths, each of which starts and ends in the same vertices as the arc of $F$ it replaces.
\begin{theorem}[Directed grid theorem, Kawarabayashi and Kreutzer 2015,~\cite{MR3388245}]\label{thm:grid}
There exists a function $g:\mathbb{N}\rightarrow \mathbb{N}$ such that for every $k\in \mathbb{N}$, every digraph $D$ with $\mathrm{dtw}(D)\ge g(k)$ contains a subdivision of $W_k$. 
\end{theorem}

Combining our Theorem~\ref{thm:main2} with Theorem~\ref{thm:grid} directly yields the following new result.

\begin{corollary}\label{cor:wall}
There exists a function $f:\mathbb{N}\rightarrow \mathbb{N}$ such that every regular digraph $D$ with degree at least $f(k)$ contains a subdivision of $W_k$.
\end{corollary}

Let us say that a digraph $F$ is a \emph{topological wall minor} if there exists a subdivision of $F$ which is isomorphic to a subdigraph of $W_k$ for some $k\in \mathbb{N}$. This forms a rich class of planar digraphs of maximum degree at most $3$, including arbitrary orientations of paths and cycles (a precise characterization of this class can be derived from the recent paper~\cite{MR4657706}). Pause to note that Corollary~\ref{cor:wall} has the following consequence.
\begin{corollary}\label{cor:top}
For every topological wall minor $F$ there exists a constant $K=K(F)$ such that every regular digraph of degree at least $K$ contains a subdivision of $F$.  
\end{corollary}
We include Corollary~\ref{cor:top} since it makes progress on the notoriously difficult problem of finding degree conditions for digraphs that guarantee their containment of a subdivision of a fixed digraph: While a classical result, independently due to Bollob\'{a}s-Thomason~\cite{MR1657911} and Koml\'{o}s-Szemer\'{e}di~\cite{MR1288443,MR1395694} guarantees for each graph $F$ the existence of a number $K=K(F)$ such that every graph of minimum degree at least $K$ contains a subdivision of $F$, the analogous result is false in full generality for digraphs: A construction of Thomassen~\cite{MR793491} (and also Remark~\ref{rem:outindtw}) imply that there exist digraphs of arbitrarily large minimum out- and in-degree which do not contain a subdivision of the complete digraph $\bivec{K}_3$ on $3$ vertices. In response to this construction, Mader~\cite{MR815582} made the following well-known conjecture.
\begin{conjecture}[Mader, 1985~\cite{MR815582}]\label{con:mader}
For every acyclic digraph $F$ there exists a number $K=K(F)$ such that every digraph $D$ with $\delta^+(D)\ge K$ contains a subdivision of $F$.
\end{conjecture}
Despite being a natural generalization of the aforementioned results of Bollob\'{a}s-Thomason and Koml\'{o}s-Szemer\'{e}di to directed graphs and having received a significant amount of attention~\cite{MR3992294,MR4757005, MR4155283,MR4543727,MR2370179,MR3906644,MR1364025,MR1377610}, Conjecture~\ref{con:mader} remains known only for precious few digraphs $F$, such as arbitrary orientations of paths and cycles, certain orientations of trees and acyclic digraphs on at most four vertices. Corollary~\ref{cor:top} shows that subdivisions of digraphs in a significantly richer class of (not necessarily acyclic) digraphs can be guaranteed in regular digraphs.

\textbf{Organization and proof overview.} In Section~\ref{sec:aux} we collect several useful observations and prove the key lemmas for the proof of our main results. In Section~\ref{sec:proof} we then give the proofs of Theorem~\ref{thm:main}, Proposition~\ref{prop:upper} and Theorem~\ref{thm:main2}. Finally, in Section~\ref{sec:conc} we conclude with open problems.

Even though our proofs of Theorems~\ref{thm:main} and~\ref{thm:main2} are short, we deem it helpful to give a high-level overview and an explanation of our proof strategies, which are atypical for directed graph theory.

Coming from undirected graph theory, a natural approach to proving results like Theorem~\ref{thm:main} and~\ref{thm:main2} is to use a foundational result of Mader~\cite{MR306050} to pass from the original (dense) graph to a highly connected subgraph, and then use the high connectivity of this subgraph to build the desired substructure or to certify high tree-width in this subgraph. In fact, this exact strategy gives straightforward proofs that every undirected graph of average degree at least $r$ contains $\Omega(r)$ openly disjoint cycles through a common vertex and has tree-width $\Omega(r)$.

However, such strategies almost never work for analogous embedding problems in directed graphs, since in order to build a \emph{strongly connected} substructure such as openly disjoint directed cycles through a vertex or a certificate for high directed tree-width, one would want to guarantee a subdigraph of high \emph{strong} vertex-connectivity. This, however, turns out be impossible, even when assuming regularity: For example (and this just one of many examples of negative results in this direction for directed graphs), Mader~\cite{MR2588727} constructed for every $r\in \mathbb{N}$ an $r$-regular digraph without any strongly $2$-arc-connected subdigraphs. Therefore,  such proof approaches seem hopeless for embedding strongly connected structures in digraphs.

With this in mind, our proof strategies for Theorems~\ref{thm:main} and~\ref{thm:main2} are quite non-standard for the study of strongly connected structures in digraphs (and this is perhaps the most important conceptual contribution of the paper the hand): Starting from an $r$-regular digraph $D$, we pursue a purely density-based approach to produce a subdigraph $D'$ which is still rather dense and has no small \emph{undirected} edge cuts between reasonably large vertex-sets (Lemma~\ref{lem:edgeconnectivity}). However we cannot guarantee any good lower bound on the \emph{strong} connectivity of $D'$, which in fact may be close to being acyclic and not even strongly $2$-connected. Hence, just within $D'$ we will in general not be able to directly find the desired collection of openly disjoint directed cycles through a vertex or a certificate for high directed tree-width. Instead, we will use $D'$ solely for the purpose of finding a ``good'' special vertex $v$ with large in- and out-degree, even after possibly deleting a small set of vertices from our digraph $D$. Crucially, we can then guarantee undirected edge-connectivity $\Omega(r^2)$ between the out- and in-neighborhood of $v$ in $D'$ (and hence also in $D$). We then go back to considering the whole of $D$. Using the fact that $D$ is an Eulerian digraph and Menger's theorem we can then essentially turn undirected edge-connectivity $\Omega(r^2)$ between the out- and in-neighborhood of $v$ into  strong vertex-connectivity $\Omega(r)$ between the same. We then use this to show that there exists a large collection of openly disjoint directed cycles through $v$, respectively to certify high directed tree-width of $D$ by identifying a so-called \emph{linked set} (the latter statement is somewhat oversimplified, but gives the correct intuition). 

We believe that this new high-level approach may be of independent interest and has potential to lead to further new results on finding more general classes of subdivisions and minors in regular digraphs of high degree.

\medskip

\textbf{Notation and Terminology.}
Given $n\in \mathbb{N}$ we denote by $[n]=\{1,\ldots,n\}$ the set of the first $n$ natural numbers. For an undirected graph $G$, as usual, we denote by $V(G)$ its vertex- and by $E(G)$ its edge-set.
Given a digraph $D$, we denote by $V(D)$ its vertex-set and by $A(D)\subseteq \{(u,v)\in V(D)^2|u\neq v\}$ the set of its arcs (where for an arc $(u,v)$ we think of $u$ as its tail/starting point and $v$ as its head/endpoint). We also use $\mathrm{v}(D):=|V(D)|$ and $\mathrm{a}(D):=|A(D)|$ as a shorthand for the number of vertices and arcs, respectively. For a vertex $v\in V(D)$, we denote by $N_D^+(v), N_D^-(v)$ its out- and in-neighborhood, respectively, and by $d_D^+(v):=|N_D^+(v)|, d_D^-(v):=|N_D^-(v)|$ its out- and in-degree. We also denote by $\delta^+(D)$, $\delta^-(D)$, $\Delta^+(D)$, $\Delta^-(D)$ the minimum out-, minimum in-, maximum out- and maximum in-degree of $D$, respectively.

Given a subset of vertices $X\subseteq V(D)$, we denote by $D[X]$ the subdigraph of $D$ induced by $X$ and by $D-X:=D[V(D)\setminus X]$ the subdigraph obtained by deleting $X$. We also write $D-v$ in case $X=\{v\}$ is a single vertex. Finally, for two disjoint subsets $X,Y$ of vertices of a digraph $D$ we denote by $\mathrm{a}_D(X,Y)$ the number of arcs in $D$ starting in $X$ and ending in $Y$, and we denote by $\mathrm{a}_D[X,Y]=\mathrm{a}_D[Y,X]:=\mathrm{a}_D(X,Y)+\mathrm{a}_D(Y,X)$ the number of all arcs in $D$ spanned between $X$ and $Y$ (in any direction).

\section{Auxiliary results}\label{sec:aux}
In this section, we collect several useful observations and key lemmas that prepare the proofs of our main results. We start by recalling the classic Menger-theorem in directed graphs, concretely the \emph{set-version} for \emph{vertex-disjoint} directed paths.
\begin{lemma}[Menger's theorem---set version]\label{lem:menger}
Let $D$ be a digraph, $U,W\subseteq V(D)$ and $k\in \mathbb{N}$. Then exactly one of the following two statements holds.
\begin{itemize}
    \item There exist $k$ vertex-disjoint directed paths in $D$, each with first vertex in $U$ and last vertex in $W$.
    \item There is a set $S\subseteq V(D)$ with $|S|\le k-1$ such that there exists no directed path in $D-S$ starting in $U$ and ending in $W$.
\end{itemize}
\end{lemma}
\noindent We record the following simple consequence of Menger's theorem, which we will use repeatedly.
\begin{corollary}\label{cor:sep}
    Let $D$ be a digraph. Then for every $v\in V(D)$ there exists some $S\subseteq V(D)$ with $|S|\le c(D)$ and a partition $(A,B)$ of $V(D)\setminus S$ such that $N_D^+(v)\subseteq A\cup S, N_D^-(v)\subseteq B\cup S$ and such that there are no arcs in $D$ which start in $A$ and end in $B$.
\end{corollary}
\begin{proof}
    Apply Lemma~\ref{lem:menger} to $D$ with $U:=N_D^+(v)$, $W:=N_D^-(v)$ and $k:=c(D)+1$. 
    
    Suppose first that the first outcome of Lemma~\ref{lem:menger} holds, i.e. there exists a family of $k$ disjoint directed paths $P_1,\ldots,P_k$ in $D$ starting in $U$ and ending in $W$. By moving to appropriate subpaths, we may w.l.o.g. assume that for each path $P_i$ its first vertex is its only vertex in $U$ and its last vertex is its only vertex in $W$. This, in particular, implies that $v\notin V(P_i)$ for each $i$, as otherwise the out-neighbor of $v$ on the path would be a member of $N_D^+(v)=U$ distinct from the first vertex of $P_i$, a contradiction. It follows from this that for each $i\in [k]$ we can obtain a directed cycle $C_i$ from $P_i$ by adding to it the vertex $v$ and the arcs from $v$ to the first vertex of $P_i$ as well as from the last vertex of $P_i$ to $v$. Also, since $P_1,\ldots,P_k$ are pairwise vertex-disjoint, it follows that $C_1,\ldots,C_k$ are a family of $k>c(D)$ openly disjoint directed cycles through the common vertex~$v$, contradicting the definition of $c(D)$. Hence, the first outcome of Lemma~\ref{lem:menger} is impossible, and so the second one has to hold: There exists a set $S\subseteq V(D)$ with $|S|\le k-1=c(D)$ such that there is no directed path in $D-S$ starting in $U$ and ending in $W$. Now let $A$ be defined as the set of all vertices that can be reached in $D-S$ via a directed path starting in $U$. Then clearly $A\supseteq U\setminus S=N_D^+(v)\setminus S$ and hence $A\cup S\supseteq N_D^+(v)$. Furthermore, since no directed path in $D-S$ can start in $U$ and end in $W$, we have $A\cap W=\emptyset$. Defining $B:=V(D)\setminus (A\cup S)$ this implies that $B\cup S=V(D)\setminus A\supseteq W=N_D^-(v)$. Furthermore, the definition of $A$ directly implies that no arc in $D-S$ leaves $A$, i.e., there is no arc in $D$ from $A$ to $B$. Hence $(A,B)$ is a partition of $V(D)\setminus S$ with the desired properties, concluding the proof of the corollary.
\end{proof}

For the proof of Theorem~\ref{thm:main2}, it will be useful to introduce the definition of \emph{linked sets}, which constitutes one of several ways of witnessing that a digraph has large directed tree-width.

\begin{definition}[cf.~\cite{MR1966425}]
Let $D$ be a digraph and $k\in  \mathbb{N}$. A subset of vertices $L\subseteq V(D)$ is called \emph{$k$-linked} if for every $S\subseteq V(D)$ with $|S|<k$ there is a strongly connected component of the digraph $D-S$ containing more than half of the vertices\footnote{Note that there can be only one strongly connected component of $D$ with this property, so this strongly connected component is uniquely defined.} of $L$.
\end{definition}

In the proof of Theorem~\ref{thm:main2} we will require the following lower bound on the directed tree-width in terms of linked sets.

\begin{lemma}[cf.~\cite{MR1966425}]\label{lem:linked}
Every digraph $D$ containing a $k$-linked set satisfies $\mathrm{dtw}(D)\ge k-1$.
\end{lemma}

Reed~\cite{MR1966425} claims that Lemma~\ref{lem:linked} was proved by Johnson et al.~\cite{MR1828440}, however we could not find this statement explicitly in their paper. However, they prove a more general statement involving another notion called \emph{havens}, from which Lemma~\ref{lem:linked} can be directly deduced. For the sake completeness, we shall include this very short deduction here.

\begin{definition}[cf.~\cite{MR1828440}]
Let $D$ be a digraph and $k\in \mathbb{N}$. A \emph{haven of order $k$} in $D$ is a function $\rho$ assigning to every subset $S$ of $V(D)$ with $|S|<k$ a strongly connected component $\rho(S)$ of $D-S$ such that the following holds.

For every pair of sets $S\subseteq S'\subseteq V(D)$ with $|S|\le |S'|<k$ we have $\rho(S)\supseteq \rho(S')$.
\end{definition}
The following statement was explicitly proved by Johnson et al.
\begin{lemma}[Johnson et al., cf.~Theorem 3.1 in~\cite{MR1828440}]\label{lem:johnson1}
Every digraph $D$ with a haven of order $k$ satisfies $\mathrm{dtw}(D)\ge k-1$.  
\end{lemma}
We can now prove Lemma~\ref{lem:linked}. 
\begin{proof}[Proof of Lemma~\ref{lem:linked}]
Let $L$ be a $k$-linked set in $D$. For any subset $S\subseteq V(D)$ with $|S|<k$, let $\rho(S)$ be defined as the (unique) strongly connected component of $D-S$ containing more than half of the vertices of $L$. We claim that $\rho$ is a haven of order $k$, and then Lemma~\ref{lem:johnson1} will imply that $\mathrm{dtw}(D)\ge k-1$, as desired. So let $S\subseteq S'\subseteq V(D)$ with $|S|\le |S'|<k$ be given. Then clearly the strongly connected component $\rho(S')$ of $D-S'$ is a strongly connected subdigraph of $D-S$, and hence is contained in a strongly connected component of $D-S$, which we call $H$. By definition of $\rho$ we have that $\rho(S')$ contains more than half of the vertices in $L$, and hence the same is true for $H$. But then (since strongly connected components are disjoint) it has to be the unique strongly connected component of $D-S$ with this property, and it follows that $\rho(S)=H$ by definition of $\rho$. Hence we have $\rho(S)=H\supseteq \rho(S')$, as desired. This shows that $\rho$ is indeed a haven of order $k$ for $D$ and concludes the proof.
\end{proof}

\noindent Finally, we will also need the following statement, showing that the directed tree-width of the symmetric orientation of an undirected graph coincides with its undirected tree-width.
\begin{lemma}[cf.~Theorem~2.1 in~\cite{MR1828440}]\label{lem:symmetric}
Let $F$ be an undirected graph, and let $D$ be the digraph with $V(D)=V(F)$ and $A(D)=\{(u,v),(v,u)|uv\in E(F)\}$. Then $\mathrm{dtw}(D)=\mathrm{tw}(F)$.
\end{lemma}

This finishes our discussion of the auxiliary results from the literature. In the remainder of this section we will prove two key lemmas for our main results.

Throughout the remainder of this paper, we 
will repeatedly use the following key definition: Given a natural number $r\in\mathbb{N}$ and reals $\beta, \gamma>0$, a digraph $D$ will be called $(r,\beta,\gamma)$-\emph{dense} if $\mathrm{v}(D)\ge \gamma r$ and $\mathrm{a}(D)\ge r\cdot \mathrm{v}(D)-\beta r^2$. The following statement allows to pass from any $(r,\beta,\gamma)$-dense digraph to an $(r,\beta,\gamma)$-dense subdigraph with decent undirected edge-connectivity between reasonably large sets of vertices.
\begin{lemma}\label{lem:edgeconnectivity}
Let $r\in \mathbb{N}$ and $\beta,\gamma>0$. Then every $(r,\beta,\gamma)$-dense digraph $D$ contains an $(r,\beta,\gamma)$-dense subdigraph $D'\subseteq D$ such that for every partition $(X,Y)$ of $V(D')$ with $|X|, |Y|\ge \gamma r$, it holds that $\mathrm{a}_{D'}[X,Y]> \beta r^2$. 
\end{lemma}
\begin{proof}
Let $D'\subseteq D$ be chosen among all $(r,\beta,\gamma)$-dense subdigraphs of $D$ such that $\mathrm{v}(D')$ is minimized. Towards a contradiction, suppose there exists a partition $(X,Y)$ of $V(D')$ with $|X|,|Y|\ge \gamma r$ such that $\mathrm{a}_{D'}[X,Y]\le\beta r^2$. Since the subdigraphs $D'[X]$ and $D'[Y]$ of $D$ both have at least $\gamma r$ vertices but cannot be $(r,\beta,\gamma)$-dense due to our choice of $D'$, it follows that $\mathrm{a}(D'[X])<r|X|-\beta r^2$, $\mathrm{a}(D'[Y])<r|Y|-\beta r^2$. Hence, we find that
$$\mathrm{a}(D')=\mathrm{a}(D'[X])+\mathrm{a}(D'[Y])+\mathrm{a}_{D'}[X,Y]<(r|X|-\beta r^2)+(r|Y|-\beta r^2)+\beta r^2=r\mathrm{v}(D')-\beta r^2,$$ contradicting that $D'$ is $(r,\beta,\gamma)$-dense. This concludes the proof of the lemma.
\end{proof}
The next observation gives a simple upper bound on the number of arcs for digraphs with small $c(D)$ or small $\mathrm{dtw}(D)$, which will later be useful in cases when $D$ has few vertices.

\begin{observation}\label{obs:densebound}
For every digraph $D$, we have $\mathrm{a}(D)< \frac{\mathrm{v}(D)(\mathrm{v}(D)+\min\{c(D),2\mathrm{dtw}(D)\})}{2}$.
\end{observation}
\begin{proof}
Let $F$ be the undirected graph with $V(F):=V(D)$ and $$E(F):=\{xy|(x,y),(y,x)\in A(D)\}.$$ Then by definition of $c(D)$, we can see that $F$ has maximum degree at most $c(D)$, and hence by the handshake-lemma at most $\frac{c(D)\mathrm{v}(D)}{2}$ edges. Furthermore, since the tree-width of $F$ equals the directed tree-width of its symmetric orientation by Lemma~\ref{lem:symmetric}, and since the latter is a subdigraph of $D$ by definition, we find that $\mathrm{tw}(F)\le \mathrm{dtw}(D)$. Lemma~\ref{lem:twedgebound} now implies that $F$ has at most $\mathrm{dtw}(D)\mathrm{v}(D)$ edges. All in all, these two facts imply that $|E(F)|\le \frac{\min\{c(D),2\mathrm{dtw}(D)\}}{2}\mathrm{v}(D)$. 

It follows directly from this that 
$$\mathrm{a}(D)\le \left(\binom{\mathrm{v}(D)}{2}-|E(F)|\right)+2|E(F)|<\frac{\mathrm{v}(D)^2}{2}+|E(F)|\le \frac{\mathrm{v}(D)(\mathrm{v}(D)+\min\{c(D),2\mathrm{dtw}(D)\})}{2},$$ as desired.
\end{proof}
The following is a key technical lemma which guarantees a vertex of high out- and in-degree in any $(r,\beta,\gamma)$-dense digraph with bounded maximum degree and small $c(D)$ or small $\mathrm{dtw}(D)$.
\begin{lemma}\label{lem:highoutandin}
Let $r\in \mathbb{N}$, and let $\alpha,\beta,\gamma>0$ be such that $2\beta<\left(1-\frac{\alpha}{2}\right)^2$ and $$\gamma\ge \left(1-\frac{\alpha}{2}\right)-\sqrt{\left(1-\frac{\alpha}{2}\right)^2-2\beta}.$$ Then for every $(r,\beta,\gamma)$-dense digraph $D$ with $\Delta^+(D), \Delta^-(D)\le r$ and $\min\{c(D),2\mathrm{dtw}(D)\}<\alpha r$ we have $\mathrm{v}(D)\ge \left(\left(1-\frac{\alpha}{2}\right)+\sqrt{\left(1-\frac{\alpha}{2}\right)^2-2\beta}\right)r$.
Suppose further that $0<\delta<\min\{1/2,1-\sqrt{\beta}\}$ is such that
$$\left(1-\frac{\alpha}{2}\right)+\sqrt{\left(1-\frac{\alpha}{2}\right)^2-2\beta}\ge 2((1-\delta)-\sqrt{(1-\delta)^2-\beta}) $$ and
$$\frac{\beta}{1/2-\delta}\le 2((1-\delta)+\sqrt{(1-\delta)^2-\beta}).$$
Then $D$ contains a vertex $v$ such that $$\min\{d_D^+(v),d_D^-(v)\}\ge \delta r.$$
\end{lemma}
\begin{proof}
We start by proving the first part of the lemma concerning the number of vertices of $D$. Throughout the rest of this proof, it will be convenient to set $x:=\frac{\mathrm{v}(D)}{r}$, and our first goal is to show $x\ge \left(1-\frac{\alpha}{2}\right)+\sqrt{\left(1-\frac{\alpha}{2}\right)^2-2\beta}$. 

Recalling that $\min\{c(D),2\mathrm{dtw}(D)\}<\alpha r$ by assumption, Observation~\ref{obs:densebound} yields that
$$r\cdot \mathrm{v}(D)-\beta r^2\le \mathrm{a}(D)<\frac{\mathrm{v}(D)(\mathrm{v}(D)+\min\{c(D),2\mathrm{dtw}(D)\})}{2}<\frac{1}{2}\mathrm{v}(D)^2+\frac{\alpha r}{2}\mathrm{v}(D).$$

Dividing the previous inequality by $r^2$ we obtain:
$$x-\beta<\frac{1}{2}x^2+\frac{\alpha}{2}x.$$

Rearranging, we find that
$$x^2+(\alpha-2)x>-2\beta.$$
This can be further simplified to 
$$\left(x-\left(1-\frac{\alpha}{2}\right)\right)^2>\left(1-\frac{\alpha}{2}\right)^2-2\beta.$$ This implies that $x$ does not lie in the interval 
$$\left[\left(1-\frac{\alpha}{2}\right)-\sqrt{\left(1-\frac{\alpha}{2}\right)^2-2\beta},\left(1-\frac{\alpha}{2}\right)+\sqrt{\left(1-\frac{\alpha}{2}\right)^2-2\beta}\right].$$ Since $D$ is $(r,\beta,\gamma)$-dense, we have $x\ge \gamma\ge \left(1-\frac{\alpha}{2}\right)-\sqrt{\left(1-\frac{\alpha}{2}\right)^2-2\beta}$, and hence it follows that in fact, $x>\left(1-\frac{\alpha}{2}\right)+\sqrt{\left(1-\frac{\alpha}{2}\right)^2-2\beta}$, as desired. Hence $\mathrm{v}(D)>\left(\left(1-\frac{\alpha}{2}\right)-\sqrt{\left(1-\frac{\alpha}{2}\right)^2-2\beta}\right)r$, concluding the proof of the first part of the lemma.

Now suppose that $\delta\in (0,1/2)$ satisfies the inequalities stated in the second part of the lemma, and let us show that there exists a vertex $v$ in $D$ with out- and in-degree at least $\delta r$.
Towards a contradiction, suppose this is not the case, i.e. every vertex $v\in V(D)$ satisfies $d_D^+(v)<\delta r$ or $d_D^-(v)<\delta r$. Let $A, B$ be a partition of $V(D)$ such that $d_D^+(a)<\delta r$ for every $a\in A$ and $d_D^-(b)<\delta r$ for every $b\in B$. Since every arc in $D$ starts in $A$, or ends in $B$, or starts in $B$ and ends in $A$, we find that
$$r\cdot \mathrm{v}(D)-\beta r^2\le \mathrm{a}(D)\le \sum_{a\in A}d_D^+(a)+\sum_{b\in B}d_D^-(b)+\mathrm{a}_D(B,A)$$
$$<\delta r\cdot (|A|+|B|)+\min\{\Delta^+(D)\cdot |B|,\Delta^-(D)\cdot |A|\}$$
$$\le \delta r \cdot \mathrm{v}(D)+\frac{r\mathrm{v}(D)}{2}=\left(\delta+\frac{1}{2}\right)r\cdot \mathrm{v}(D).$$

Rearranging yields $\mathrm{v}(D)\le \frac{\beta}{\frac{1}{2}-\delta}r$, i.e. $x\le \frac{\beta}{\frac{1}{2}-\delta}$. In particular, $$x\in \left[\left(1-\frac{\alpha}{2}\right)+\sqrt{\left(1-\frac{\alpha}{2}\right)^2-2\beta},\frac{\beta}{\frac{1}{2}-\delta}\right].$$

Redoing the first steps of the previous estimate, but bounding $\mathrm{a}_D(B,A)$ now by the trivial upper bound $|A||B|\le \left(\frac{|A|+|B|}{2}\right)^2=\frac{1}{4}\mathrm{v}(D)^2$, we next obtain that
$$r\cdot \mathrm{v}(D)-\beta r^2\le \mathrm{a}(D)< \delta r\cdot \mathrm{v}(D)+\frac{1}{4}\mathrm{v}(D)^2.$$
Dividing by $r^2$, we obtain $x-\beta< \delta x+\frac{1}{4}x^2$. Rearranging now yields $(x-2(1-\delta))^2>4((1-\delta)^2-\beta)$. Taking roots, we obtain $|x-2(1-\delta)|>2\sqrt{(1-\delta)^2-\beta}$. This implies that $x$ does not lie in the interval $[2((1-\delta)-\sqrt{(1-\delta)^2-\beta}),2((1-\delta)+\sqrt{(1-\delta)^2-\beta})]$. This, however, is a contradiction, since we have previously shown that $$x\in \left[\left(1-\frac{\alpha}{2}\right)+\sqrt{\left(1-\frac{\alpha}{2}\right)^2-2\beta},\frac{\beta}{\frac{1}{2}-\delta}\right]\subseteq [2((1-\delta)-\sqrt{(1-\delta)^2-\beta}),2((1-\delta)+\sqrt{(1-\delta)^2-\beta})].$$ This shows that our above assumption was wrong: $D$ must indeed contain a vertex $v$ with out- and in-degree at least $\delta r$. This concludes the proof of the second part of the lemma. 
\end{proof}
\section{Proofs of Theorem~\ref{thm:main}, Proposition~\ref{prop:upper} and Theorem~\ref{thm:main2}}\label{sec:proof}
With the auxiliary results from the previous section at hand, we are now ready to present the proofs of our main results. We start with Theorem~\ref{thm:main}.
\begin{proof}[Proof of Theorem~\ref{thm:main}]
Throughout this proof, we fix the parameters $\alpha:=\frac{3}{22}, \beta:=\frac{3}{11},$ $\delta:=\gamma:=\frac{4}{11}$.
Let $r\in \mathbb{N}$ and an $r$-regular digraph $D$ be given to us. We have to show that $c(D)\ge \lceil \alpha r\rceil$. Towards a contradiction, suppose that $c(D)<\alpha r$. Note that since $D$ is $r$-regular, it satisfies $\mathrm{v}(D)\ge r+1\ge \gamma r$ and $\mathrm{a}(D)=\sum_{v\in V(D)}d_D^+(v)=r\cdot \mathrm{v}(D)\ge r\cdot \mathrm{v}(D)-\beta r^2$. Hence, $D$ is $(r,\beta,\gamma)$-dense. Hence, by Lemma~\ref{lem:edgeconnectivity} there exists an $(r,\beta,\gamma)$-dense subdigraph $D'$ of $D$ such that for every partition $(X,Y)$ of $V(D')$ with $|X|, |Y|\ge \gamma r$, it holds that $\mathrm{a}_{D'}[X,Y]> \beta r^2$. We clearly have that $\Delta^+(D')\le \Delta^+(D)=r$, $\Delta^-(D')\le \Delta^-(D)=r$ and $c(D')\le c(D)<\alpha r$ by assumption. Pause to verify that the parameters $\alpha,\beta,\gamma,\delta$ as defined above satisfy all pre-conditions of Lemma~\ref{lem:highoutandin}. Hence, we may apply Lemma~\ref{lem:highoutandin} to $D'$ and find that $\mathrm{v}(D')\ge \frac{3}{2}r$ and that there exists a vertex $v\in V(D')$ with $d_{D'}^+(v), d_{D'}^-(v)\ge \delta r=\gamma r$. 
We next apply Corollary~\ref{cor:sep} to the vertex $v$ in the digraph $D$, and obtain a set $S\subseteq V(D)$ with $|S|\le c(D)<\alpha r$ and a partition $(A,B)$ of $V(D)\setminus S$ such that there are no arcs from $A$ to $B$ in~$D$, and such that $N_D^+(v)\subseteq A\cup S, N_D^-(v)\subseteq B\cup S$. 

Now, let us define $A':=A\cap V(D')$, $B':=B\cap V(D')$ an $S':=S\cap V(D')$. We then have $|A'|+|B'|= \mathrm{v}(D')-|S'|\ge \mathrm{v}(D')-|S|>\frac{3}{2}r-\alpha r>2\gamma r$. Hence, we have $|A'|>\gamma r$ or $|B'|>\gamma r$. 

Let us now define a partition $(X,Y)$ of $V(D')$ as $X:=A'$ and $Y:=B'\cup S'$ if $|A'|>\gamma r$, and as $X:=A'\cup S'$ and $Y:=B'$ if $|A'|\le \gamma r$ (and thus $|B'|>\gamma r$). 

Observe that $A'\cup S'\supseteq N_{D'}^+(v)$ and $B'\cup S'\supseteq N_{D'}^-(v)$. From this and the fact that $d_{D'}^+(v), d_{D'}^-(v)\ge \gamma r$ we can see that the so-defined partition $(X,Y)$ satisfies $|X|, |Y|\ge \gamma r$ in each case. We thus find that $\mathrm{a}_{D'}[X,Y]> \beta r^2$. 

Moving on, let us first consider the case that $(X,Y)=(A'\cup S',B')$. We then have that every arc between $X$ and $Y$ in $D'$ is also an arc between $A\cup S$ and $B$ in $D$. Also note that, since $D$ is $r$-regular and hence Eulerian, we have
$\mathrm{a}_D[A\cup S,B]=2\mathrm{a}_D(A\cup S,B)$. All in all, it follows that

$$\beta r^2\le \mathrm{a}_{D'}[X,Y]\le \mathrm{a}_D[A\cup S, B]=2\mathrm{a}_D(A\cup S, B)=2\mathrm{a}_D(S,B)\le 2r|S|,$$
 where we used that there are no arcs from $A$ to $B$ in the penultimate step. This implies $|S|\ge \frac{\beta }{2}r=\alpha r$, a contradiction to our choice of $S$, finishing the proof in this first case.

Second, let us consider the case that $(X,Y)=(A',B'\cup S')$. We then have that every arc between $X$ and $Y$ in $D'$ is also an arc between $A$ and $B\cup S$ in $D$. Also note that, since $D$ is $r$-regular and hence Eulerian, we have
$\mathrm{a}_D[A,B\cup S]=2\mathrm{a}_D(A,B\cup S)$. Similar as in the previous case, it follows that

$$\beta r^2\le \mathrm{a}_{D'}[X,Y]\le \mathrm{a}_D[A, B\cup S]=2\mathrm{a}_D(A, B\cup S)=2\mathrm{a}_D(A,S)\le 2r|S|,$$
 where we again used that there are no arcs from $A$ to $B$ in the penultimate step. This again implies $|S|\ge \frac{\beta }{2}r=\alpha r$, contradicting our choice of $S$.

 Having reached a contradiction in both cases, we conclude that our initial assumption, namely that $c(D)<\alpha r$, must have been wrong. Hence, it follows that $c(D)\ge \lceil\alpha r\rceil$, as desired. This concludes the proof of the theorem.
\end{proof}

Using Corollary~\ref{cor:sep} from the previous section, we can now also supply the missing proof of Proposition~\ref{prop:upper} (which then implies the upper bound on $c_r$ in Corollary~\ref{cor:upper}).

\begin{proof}[Proof of Proposition~\ref{prop:upper}]
Let $r,b\in \mathbb{N}$ be given. Let $D$ be an $r$-regular digraph such that $c(D)=c_r$.

We now construct a digraph $D'$ on vertex-set $V(D'):=V(D)\times [b]$ as follows: There is an arc form a vertex $(u,i)$ to a vertex $(v,j)$ in $D'$ if and only if $(u,v)\in A(D)$, for any choice of $i,j\in [b]$. Intuitively, $D'$ is the blow-up of $D$ obtained by replacing each vertex $v$ by a corresponding independent set $\{v\}\times [b]$ of size $b$ and placing a directed complete bipartite graph between any two blowup-sets previously connected by an arc. It is immediately verified that $D'$ is an $rb$-regular digraph. 

We now claim that $c(D')\le c_r\cdot b$ (and hence $c_{rb}\le c_r\cdot b$, which is what we want to show). 

Indeed, towards a contradiction suppose that $c(D')\ge c_rb+1$. Then there exists a vertex $(v,i)\in V(D')$ and a collection $C_1,\ldots,C_{c_rb+1}$ of $c_rb+1$ openly disjoint directed cycles through $(v,i)$ in $D'$. On the other hand, Corollary~\ref{cor:sep} implies that there exists a set $S\subseteq V(D)$ with $|S|\le c(D)=c_r$ such that there exists no directed path in $D-S$ starting in $N_D^+(v)$ and ending in $N_D^-(v)$. In particular, there exists no directed walk in $D-S$ starting in $N_D^+(v)$ and ending in $N_D^-(v)$. Now note that since $|S\times [b]|=|S|b\le c_r\cdot b$ and since the sets $V(C_j)\setminus \{(v,i)\}, j=1,\ldots,c_rb+1$ are pairwise disjoint, there must exist some $j\in [c_rb+1]$ such that $(V(C_j)\setminus \{(v,i)\})\cap (S\times [b])=\emptyset$. It is now easy to see that the projection of the vertex-sequence of the directed path $C_j-(v,i)$ in $D'$ to the sequence consisting of the corresponding vertices in $D$ forms a directed walk in $D$ which starts in $N_D^+(v)$, ends in $N_D^-(v)$ and does not use any vertices in $S$. This is a contradiction to what we stated above, and this concludes the proof. Hence indeed we have $c_{rb}\le c(D')\le c_r\cdot b$, as desired.
\end{proof}
Finally, we present the proof of our second main result, Theorem~\ref{thm:main2}, which uses a similar technique as the proof of Theorem~\ref{thm:main}.
\begin{proof}[Proof of Theorem~\ref{thm:main2}]
Throughout this proof, let us fix the parameters $\alpha:=0.1, \beta:=0.1, \gamma:=0.3, \alpha':=\alpha=0.1, \beta':=0.15, \gamma':=1.7$ and $\delta':=\gamma=0.3$.

Let $r\in \mathbb{N}$ and let $D$ be an $r$-regular digraph. Our goal is to show that $\mathrm{dtw}(D)\ge \left\lfloor\frac{r}{20}\right\rfloor=\lfloor \frac{\alpha}{2}r\rfloor$. Suppose towards a contradiction that $\mathrm{dtw}(D)\le\lfloor \frac{\alpha}{2}r\rfloor-1<\frac{\alpha}{2}r$.

Since $D$ is $r$-regular and hence satisfies $\mathrm{a}(D)=r\mathrm{v}(D)\ge r\mathrm{v}(D)-\beta r^2$ as well as $\mathrm{v}(D)\ge r+1\ge \gamma r$, we have that $D$ is $(r,\beta,\gamma)$-dense. Thus, by Lemma~\ref{lem:edgeconnectivity} there exists an $(r,\beta,\gamma)$-dense subdigraph $D'\subseteq D$ such that for every partition $(X,Y)$ of $V(D')$ with $|X|, |Y|\ge \gamma r$, we have that $\mathrm{a}_D[X,Y]> \beta r^2$.

Let $k:=\lfloor \frac{\alpha}{2}r\rfloor+1$ and $L:=V(D')$. By Lemma~\ref{lem:linked} we cannot have that $L$ is $k$-linked in $D$, for otherwise we would have $\mathrm{dtw}(D)\ge k-1=\lfloor \frac{\alpha}{2}r\rfloor$, contradicting our initial assumption. By definition of a $k$-linked set it now follows that there exists some set $S\subseteq V(D)$ with $|S|\le k-1\le \frac{\alpha}{2}r$ such that every strongly connected component of $D-S$ intersects $L$ in at most $\frac{|L|}{2}$ vertices. Let $D'':=D'-S$. Since $D'$ is $(r,\beta,\gamma)$-dense and since every vertex in $D'$ has out- and in-degree at most $r$, we have $$\mathrm{a}(D'')\ge \mathrm{a}(D')-2r|S\cap V(D')|\ge r\cdot \mathrm{v}(D')-2r|S\cap V(D')|-\beta r^2=r\cdot\mathrm{v}(D'')-r|S\cap V(D')|-\beta r^2$$ $$\ge r\cdot \mathrm{v}(D'')-r|S|-\beta r^2\ge r\cdot \mathrm{v}(D'')-\left(\beta+\frac{\alpha}{2}\right)r^2=r\cdot \mathrm{v}(D'')-\beta'r^2.$$

Note that $2\mathrm{dtw}(D')\le 2\mathrm{dtw}(D)<\alpha r$, $\Delta^+(D'), \Delta^-(D')\le r$  and pause to verify that $\alpha,\beta,\gamma$ satisfy the preconditions of the first part of Lemma~\ref{lem:highoutandin}. Hence, we may apply the first part of Lemma~\ref{lem:highoutandin} to $D'$ to find that $$\mathrm{v}(D')\ge \left(\left(1-\frac{\alpha}{2}\right)+\sqrt{\left(1-\frac{\alpha}{2}\right)^2-2\beta}\right)r>1.78r.$$ This implies $\mathrm{v}(D'')\ge \mathrm{v}(D')-|S|> 1.78r-\frac{\alpha}{2}r=1.73r>\gamma' r$. Recalling the lower bound on $\mathrm{a}(D'')$ established previously, we conclude that $D''$ is $(r,\beta',\gamma')$-dense. Since $D''\subseteq D'\subseteq D$, we furthermore have $2\mathrm{dtw}(D'')\le 2\mathrm{dtw}(D)<\alpha r=\alpha'r$ and $\Delta^+(D''), \Delta^-(D'')\le r$. Pause to verify that $\alpha',\beta',\gamma',\delta'$ as defined above satisfy all the necessary inequalities for an application of Lemma~\ref{lem:highoutandin} to $D''$. Hence, the lemma guarantees the existence of a vertex $v\in V(D'')=L\setminus S$ such that $d_{D''}^+(v), d_{D''}^-(v)\ge \delta' r=\gamma r$.

Let us denote by $V^+$ the set of all vertices in $D-S$ which can be reached, starting from $v$, via a directed path in $D-S$. Similarly, let $V^-$ denote the set of all vertices in $D-S$ which can reach $v$ via a directed path in $D-S$. Pause to note that $V^+\cap V^-$ is exactly the vertex-set of the unique strongly connected component of $D-S$ containing $v$. Hence, by our choice of $S$ we know that $V^+\cap V^-$ cannot contain more than half of the vertices of $L$. This implies that $|V^+\cap L|+|V^-\cap L|\le |L|+|(V^+\cap V^-)\cap L|\le \frac{3}{2}|L|$. In particular, we must have $|V^+\cap L|\le \frac{3}{4}|L|$ or $|V^-\cap L|\le \frac{3}{4}|L|$. Moving forward, we w.l.o.g. assume that $|V^+\cap L|\le \frac{3}{4}|L|$ (the case $|V^-\cap L|\le \frac{3}{4}|L|$ can be handled completely symmetrically up to arc-reversal).

Now let us define $X:=V^+\cup S$ and $Y:=V(D)\setminus (V^+\cup S)$ as well as $X':=X\cap L=X\cap V(D'), Y':=Y\cap L=Y\cap V(D')$. Then $(X,Y)$ is a partition of $V(D)$ and $(X',Y')$ a partition of $V(D')$. Note that by definition of $V^+$ and our choice of $v$, we have 
$|X'|=|X\cap L|\ge |V^+\cap L|\ge |N_{D-S}^+(v)\cap L|\ge d_{D''}^+(v)\ge \gamma r$. Furthermore, we have $$|Y'|=|Y \cap L|= |L|-|(V^+\cup S)\cap L|\ge |L|-|V^+\cap L|-|S|\ge \frac{1}{4}|L|-|S|$$
$$=\frac{\mathrm{v}(D')}{4}-|S|>\frac{1.78}{4}r-\frac{\alpha}{2}r=0.395 r>\gamma r.$$
Hence, our initial choice of $D'$ implies that $\mathrm{a}_{D'}[X',Y']> \beta r^2$. In particular, it follows that $\mathrm{a}_D[X,Y]> \beta r^2$. Recalling that $D$ is $r$-regular and hence Eulerian, we furthermore find that $\mathrm{a}_D(X,Y)=\frac{1}{2}\mathrm{a}_D[X,Y]> \frac{\beta}{2}r^2$. On the other hand, by definition of $V^+$, there can be no arc in $D$ starting in $V^+$ and ending in $Y=V(D)\setminus (V^+\cup S)$. Hence, we have $a_D(X,Y)=a_D(S,Y)\le r|S|\le \frac{\alpha}{2}r^2$. All in all, this implies that $\frac{\beta}{2}r^2<a_D(X,Y)\le \frac{\alpha}{2}r^2$, a contradiction since $\alpha=\beta=0.1$ by definition. This shows that our initial assumption in this proof that $\mathrm{dtw}(D)<\left\lfloor\frac{r}{20}\right\rfloor$ was wrong. This shows that indeed $\mathrm{dtw}(D)\ge \left\lfloor\frac{r}{20}\right\rfloor$, concluding the proof of the theorem.
\end{proof}
\section{Conclusion}\label{sec:conc}
The following  is a consequence of Theorem~\ref{thm:main} and Proposition~\ref{prop:upper} and summarizes their essence.
\begin{corollary}
The limit $\lim_{r\rightarrow \infty}\frac{c_r}{r}$ exists and lies between $\frac{3}{22}$ and $\frac{7}{8}$.
\end{corollary}
\begin{proof}
Define $a_r := \frac{c_r}{r}$. We show that $(a_r)_{r\in \mathbb{N}}$ converges. To do so, let us define $L:=\inf_{r\in \mathbb{N} }a_r$ and let us prove that $a_r$ converges to $L$. Observe that it suffices to show that $\lim\sup_{r\rightarrow \infty}a_r\le L$ and $\lim\inf_{r\rightarrow \infty}a_r\ge L$. Since the latter inequality holds trivially by definition of $L$, it suffices to prove the first. For the moment, fix some $r \in \mathbb{N}$ arbitrarily. For any $n \in \mathbb{N}$, write $n = qr + s$ with $q\in \mathbb{N}_0$ and $s\in [r]$. Since $(c_r)_{r\in \mathbb{N}}$ is increasing (as was observed by Mader~\cite{MR2588727}), we have
$
c_n = c_{qr + s} \le c_{(q+1)r}.
$
Furthermore, by Proposition~\ref{prop:upper} we have
$
c_{(q+1)r} \le (q+1)c_r.
$
Hence,
\[
a_n = \frac{c_n}{n} \le \frac{(q+1)c_r}{qr+1}.
\]
Taking the $\lim\sup$ for $n\rightarrow\infty$ (and hence $q\rightarrow \infty$) on both sides, we obtain that
\[
\limsup_{n \to \infty} a_n \le a_r.
\]
Since we initially fixed $r$ arbitrarily, it follows that
\[
\limsup_{n \to \infty} a_n \le \inf_{r \ge 1} a_r = L.
\]

This is what we wanted to show, demonstrating that the limit $\lim_{r\rightarrow \infty}\frac{c_r}{r}$ indeed exists and equals $L$. The stated lower and upper bounds on the limit now immediately follow from Theorem~\ref{thm:main} and Corollary~\ref{cor:upper}.
\end{proof}
It would be very interesting to determine the limit $\lim_{r\rightarrow\infty}\frac{c_r}{r}$ precisely. We leave this as an open problem for future work.
\begin{problem}
    Determine $\lim_{r\rightarrow\infty} \frac{c_r}{r}$.
\end{problem}
Similarly, motivated by Theorem~\ref{thm:main2} it would be interesting to determine the best-possible constant $\varepsilon>0$ such that every $r$-regular digraph $D$ has directed tree-width at least $\varepsilon r$. 
\bibliographystyle{abbrv}
\bibliography{references.bib}

\end{document}